\documentclass[a4paper,12pt,notitlepage]{amsart}
\usepackage[utf8]{inputenc}

\usepackage{amsmath}
\usepackage{amsthm}
\usepackage{amsfonts}
\usepackage{amssymb}
 \usepackage{enumerate}
\usepackage{graphicx}
\usepackage{verbatim}
\usepackage{hyperref}

\usepackage{dsfont}
\usepackage{mathrsfs} 
\usepackage{float}
\usepackage{stmaryrd}

\DeclareMathOperator{\sprod}{SP}

\newcommand{\bN}{\mathbb{N}}

\newcommand{\cP}{\mathcal{P}}
\newcommand{\cC}{\mathcal{C}}
\newcommand{\cD}{\mathcal{D}}

\newcommand{\cS}{\mathcal{S}}
\newcommand{\cT}{\mathcal{T}}
\newcommand{\cV}{\mathcal{V}}

\newcommand{\dsA}{\mathds{A}}
\newcommand{\dsE}{\mathds{E}}
\newcommand{\dsH}{\mathds{H}}

\newcommand{\dsX}{\mathds{X}}

\newcommand{\es}{\approx}
\newcommand{\ls}{\lesssim}
\newcommand{\gs}{\gtrsim}

\newcommand{\at}{\alpha_{\tau}}
\newcommand{\ato}{\alpha_{\tau}^{\text{opt}}}

\newcommand{\Ds}{D^{\ast}_N}
\newcommand{\eps}{\varepsilon}
\newcommand{\epst}{\varepsilon^{\tau}}
\newcommand{\I}{[0,1)}
\newcommand{\Psisd}{\Psi^{\textsf{sd}}}
\newcommand{\Psinsd}{\Psi^{\neg}}
\newcommand{\rsf}{\mathsf{r}}
\newcommand{\Scycle}{\Sigma_{\text{cycle}}}
\newcommand{\Stree}{\Sigma_{\text{tree}}}
\newcommand{\Sti}{\Sigma_1^{\tau}}
\newcommand{\Stii}{\Sigma_2^{\tau}}

\newcommand{\Ssl}{S_{<l}}
\newcommand{\Sgl}{S_{\geq l}}
\newcommand{\trho}{\tilde{\rho}}

\newcommand{\vr}{\vec{r}}
\newcommand{\vs}{\vec{s}}
\newcommand{\vt}{\vec{t}}

\theoremstyle{plain}		\newtheorem{theorem}{Theorem}[section]
				\newtheorem{lemma}[theorem]{Lemma}
\theoremstyle{definition}	\newtheorem{proposition}[theorem]{Proposition}
				\newtheorem{remark}[theorem]{Remark}
				\newtheorem{definition}[theorem]{Definition}

\title[On an explicit lower bound for the star discrepancy]{On an explicit lower bound for the star discrepancy in three dimensions}
\author{Florian Puchhammer}

 \address{{\bf{Florian Puchhammer}}\\
             Institute of Financial Mathematics and
             Applied Number Theory\\
			University Linz\\
 			Altenbergerstraße 69\\
 			4040 Linz\\
			AUSTRIA}
 \email{florian.puchhammer@jku.at}



\thanks{The author is supported by the Austrian Science Fund (FWF),
Project F5507-N26,
which~is a part of the Special Research Program
``Quasi-Monte Carlo Methods: Theory and Applications''\!}
\begin{document}

\begin{abstract}
Following a result of D.~Bylik and M.T.~Lacey from 2008 it is known that there exists an absolute constant $\eta>0$ such that the (unnormalized) $L^{\infty}$-norm of the three-dimensional discrepancy function, i.e. the (unnormalized) star discrepancy $D^{\ast}_N$, is bounded from below by $D_{N}^{\ast}\geq c (\log N)^{1+\eta}$, for all $N\in\mathbb{N}$ sufficiently large, where $c>0$ is some constant independent of $N$. This paper  builds upon their methods to verify that the above result holds with $\eta<1/(32+4\sqrt{41})\approx0.017357\ldots$
\end{abstract}

\maketitle
\section{Introduction and statement of the result}

Suppose we are given a set $\cP_N$ consisting of $N$ points in the $d$-dimensional unit cube. We intend to investigate how \emph{well} this set is distributed in $\I^{d}$. To this end we introduce the \emph{discrepancy function}
\begin{equation*}
 D_N(x):=N \lambda_d([0,x))-\#(\cP_N\cap[0,x)),\qquad x\in\I^{d},
\end{equation*}
i.e. the difference between the expected and actual number of points of $\cP_N$ in $[0,x)$ if we assume uniform distribution. Here, $\lambda_d$ denotes the $d$-dimensional Lebesgue measure and we abbreviated $[0,x)=[0,x_1)\times[0,x_2)\times\cdots\times[0,x_d)$ for $x=(x_1,x_2,\ldots,x_d)$. Furthermore, we refer to its $L^{\infty}$-norm
\begin{equation*}
 \Ds:=\sup_{x\in\I^{d}}|D_N(x)|
\end{equation*}
as \emph{star discrepancy}. Notice that, in other literature, this entity often appears as a normalized version, i.e. $\Ds/N$.

Over time an extensive theory has evolved around the magnitude of $\Ds$ in terms of $N$ for arbitrary as well as for specific point sets. See, for instance, the books \cite{DicDig10,MatGeo99,KuiUni74}, just to name a few. Finding the exact order of growth seems to be an intriguingly difficult question and has not yet been solved for dimensions three or higher. In this paper we focus on a lower bound for the star discrepancy of arbitrary sets of $N$ points in the three-dimensional case based on the work of D.~Bilyk and M.T.~Lacey \cite{BilOnt083}. More precisely, we show
\begin{theorem}
 \label{thm:main}
 For all $N$-point sets in $\I^{3}$, with $N$ sufficiently large, the star discrepancy satisfies
 \begin{equation*}
  \Ds\geq C(\log N)^{1+\eta}, \quad\text{for all } \eta<1/(32+4\sqrt{41})\es0.017357\ldots
 \end{equation*}
\end{theorem}
To the author's best knowledge this is the first quantitative result with respect to $\eta$.

It is worth mentioning that the basic inherent ideas reach back to K.F. Roth's seminal paper \cite{RotOni54}, in which he showed 
\begin{theorem}[Roth, 1954]
We have $\Ds\geq \|D_N\|_2\geq c_d \left(\log N\right)^{(d-1)/2}$ for all $d\geq2$.
\end{theorem}
Although this bound is now known not to be sharp for $\Ds$ (see Schmidt's theorem below) it was his approach using the system of Haar functions and Haar decompositions which struck a chord at that time and lead to a completely new methodology for proving discrepancy bounds. For a comprehensive survey see \cite{BilOnr14}, for instance. 

It took as much as 18 years until a better estimate for $\Ds$ in the two-dimensional case was discovered by W.M. Schmidt, see \cite{SchIrr72}:
\begin{theorem}[Schmidt, 1972]
\label{thm:schmidt}
 For $d=2$ we have $ \Ds\geq C \log N$.
\end{theorem}
This bound is even known to be sharp. Later, in 1981, G.~Hal{\'a}sz managed to give a proof of Schmidt's result by refining Roth's approach via introducing special auxiliary functions, namely Riesz products, and using duality, see \cite{HalOnr81}. Both, Roth's and Hal{\'a}sz' proof are also to be found in \cite{MatGeo99}. Unfortunately, Hal{\'a}sz' methods are not directly applicable to higher dimensions, due to a shortfall of certain orthogonality properties.

This shortfall leads us to yet another main ingredient of the proof of Bilyk and Lacey as well as of this paper. In \cite{BecAtw89} J.~Beck laid the groundwork for combining Hal{\'a}sz' approach to graph theory and probability theory in three dimensions. He thereby gave the first improvement to Roth's bound  by a double logarithmic factor in this case. In fact, he proved the following theorem.
\begin{theorem}[Beck, 1989]
\label{thm:beck}
For all $N$-point sets in $\I^{3}$ and for all $\eps>0$  we have
\begin{equation*}
 \Ds\geq C_{\eps} \log N \cdot(\log\log N)^{1/8-\eps}
\end{equation*}
\end{theorem}

For the sake of completeness one needs to add that an analogue of Theorem~\ref{thm:main} for arbitrary dimension $d\geq4$ was proven by Bilyk and Lacey together with A.~Vagharshakyan in \cite{BilOnt08d}. Within their paper they showed that the exponent of the logarithm in Roth's theorem can be increased to $(d-1)/2+\eta_d$ with an (unspecified) $\eta_d>0$. Due to the transition to higher dimensions and  to simplification reasons several arguments were refined and the overall strategy was slightly changed in comparison to the three-dimensional case. Apart from the increasing combinatorial effort this is one of the main reasons why the same line of reasoning as in the proof of Theorem~\ref{thm:main} would not (yet) work in higher dimensions. This might  be an interesting subject to be investigated in the future.

The author would also like to mention that a new proof for the lower bound of the star discrepancy of the first $N$ points of a sequence in the unit interval has recently been discovered by G. Larcher, see \cite{LarOnt15}, and has been slightly improved upon in \cite{LarAni16}, which transfers to two-dimensional point sets by a result from \cite{KuiUni74}.

The second section is dedicated to briefly describe the main ideas of Hal{\'a}sz' proof of Theorem~\ref{thm:schmidt} as well as to explain why his strategy cannot be directly extended to higher dimensions. This serves as an incentive to present the result of Bilyk and Lacey, i.e. Theorem~\ref{thm:main} without the specific bound for $\eta$, in Section~\ref{sec:bylik}, as they incorporate these ideas and provide the tools to fill the aforementioned gaps. We focus on one of these tools, the so-called \emph{Littlewood-Paley inequalities}, in Section~\ref{sec:lp} since they play an integral role in our proof. Finally, in Section~\ref{sec:main},  we carefully estimate the $L^1$-norm of a certain auxiliary function $\Psinsd$ which already appeared in \cite{BilOnt083}. This, in turn,  contributes the crucial bound for $\eta$ and thus completes the proof of Theorem~\ref{thm:main}.

\section{Hal{\'a}sz' proof of Theorem~\ref{thm:schmidt}}
\label{sec:halasz}

The essential idea behind this proof is to choose an auxiliary function $\Phi$ in such a way that it is complicated enough to recapture the overall structure of $D_N$ well, while, on the other hand, it remains relatively easy to handle. More precisely, one constructs $\Phi$ such that $\|\Phi\|_1\leq2$ and $\langle\Phi,D_N\rangle\geq c\log N$ for some $c>0$ since then, by duality,
\begin{equation*}
 2\Ds=2\|D_N\|_{\infty}\geq\langle\Phi,D_N\rangle\geq c\log N.
\end{equation*}
This behaviour can be achieved by using  sums of signed Haar functions.
\begin{definition}
 \label{def:haarfunction}
 Let $\cD$ denote the class of dyadic intervals, i.e.
 \begin{equation*}
  \cD=\{[a2^{-k},(a+1)2^{-k}):~k\in\bN \text{ and }0\leq a<2^{k}\}.
 \end{equation*}
Furthermore, we subdivide each $J\in\cD$ into a left and a right half, $J_{l}$ and $J_{r}$,  respectively, and define the \emph{one-dimensional Haar function} as $h_J=-\mathds{1}_{J_l}+\mathds{1}_{J_r}$. In higher dimensions $d\geq2$ we take a dyadic rectangle $R=J_1\times J_2\times\cdots\times J_d\in\cD^{d}$ and $x=(x_1,x_2,\ldots,x_d)\in\I^{d}$ and set
\begin{equation*}
 h_{R}(x)=h_{J_1}(x_1)h_{J_2}(x_2)\cdots h_{J_d}(x_d).
\end{equation*}
\end{definition}

One of the main advantages of working in this function system is that  products of Haar functions again yield Haar functions in some cases. This is indicated in the following lemma, see \cite{BilOnt083}.
\begin{lemma}[Product rule]
\label{lemma:productrule}
 Let $R_1, R_2,\ldots,R_k\in\cD^{d}$ with non-empty intersection. If, additionally, the $t$-th coordinates of all rectangles are mutually different for all $1\leq t\leq d$, then
 \begin{equation*}
  h_{R_1}h_{R_2}\cdots h_{R_k}=\sigma h_{S}, \quad\text{where }S=R_1\cap\cdots \cap R_k\text{ and }\sigma\in\{-1,+1\}.
 \end{equation*}
\end{lemma}
Let us now set 
\begin{equation}
\label{eqn:halaszrfunction}
 f_k=\sum_{R=J_1\times J_2\in\cD^2,|R|=2^{-n}, |J_1|=2^{-k}}\eps_{R}h_{R},\qquad 0\leq k\leq n,
\end{equation}
for a specific choice of signs $\eps_{R}$ which we do not want to specify here, and where $n$ is chosen such that $2^{n-2}\leq N<2^{n-1}$. Subsequently, we define $\Phi$ as the Riesz product
\begin{equation*}
 \Phi=\prod_{k=0}^{n}(1+\gamma f_k)-1 = \gamma \sum_{k=0}^{n}f_k+\Phi_{>n},\quad \gamma\in(0,1),
\end{equation*}
where $\Phi_{>n}$ contains all sums of products of functions of the form (\ref{eqn:halaszrfunction}). The key observation is that in dimension $d=2$ two or more \emph{hyperbolic} dyadic rectangles (i.e. they share the same volume) cannot coincide in any of their coordinates and, thus, their product is a Haar function again as a result of the product rule. The upper bound on the norm $\|\Phi\|_1$ can now be easily obtained with the help of Lemma~\ref{lemma:productrule} and the lower bound for $\langle D_n,\Phi\rangle$ follows from a special choice of coefficients $\eps_R$ and a standard argument involving the product rule again (see, e.g.,~\cite{MatGeo99}).

Observe that the \emph{key observation} from above deprives us of the possibility to repeat this proof verbatim for $d\geq3$. Indeed, already in dimension 3 the length of one coordinate of a hyperbolic rectangle does not fully specify the lengths of the other two, and, hence, \emph{coincidences} may occur.

\section{An outline of the strategy behind the existence result}
\label{sec:bylik}
In order to make the machinery of Hal{\'a}sz work in dimension $d=3$ and in order to improve upon Beck's result, Theorem~\ref{thm:beck}, Bilyk and Lacey had to modify the auxiliary function  on the one hand, and used more involved analytical tools adjusted to it, on the other. Also, they had to make up for the shortfall of the product rule in certain cases, as stated in the above paragraph. Since this is the part on which this paper emphasizes,  this is dealt with in the next section. We shall now turn to the construction of our auxiliary function.
\begin{definition}
 \label{def:rademacher}
 For $n\in\bN$ let
 \begin{equation*}
  \dsH_n=\left\{ \vr=(r_1,r_2,r_3)\in\bN^3:~|\vr|:=r_1+r_2+r_3=n \right\},
 \end{equation*}
 where the letter ``$\dsH$'' is used to resemble the term \emph{hyperbolic}. Two or more hyperbolic vectors have a \emph{coincidence} if their entries agree in one coordinate and are said to be \emph{strongly distinct} in the other case. Furthermore, for $\vr\in\dsH_n$ we define the set $\cD^{3}_{\vr}=\{J_1\times J_2\times J_3\in\cD^3:|J_t|=2^{-r_t}\}$ and, subsequently, call the function
 \begin{equation*}
 f_{\vr}=\sum_{R\in\cD_{\vr}^3}\alpha(R)h_R,\qquad\alpha(R)\in\{-1,1\}
\end{equation*}
an $\rsf$-\emph{function} with parameter $\vr\in\dsH_n$. Naturally, they generalize (\ref{eqn:halaszrfunction}).
\end{definition}
\begin{remark}
 \label{rem:rademacher}
These functions have mean zero and $f_{\vr}^2=\mathds{1}_{\I^3}$. Moreover, the product $f_{\vr} f_{\vs}$ gives an $\rsf$-function if $\vr,\vs\in\dsH_n$ are strongly distinct, as a consequence of the product rule. Also, products of two or more $\rsf$-functions have mean zero if the maximum of the entries of the underlying vectors is unique in some coordinate.
\end{remark}
For the rest of this paper we write $A\ls B$ if there exists an absolute constant $c$ independent of $N$ such that $A\leq c B$. Correspondingly,  $A\es B$ indicates equality up to a multiplicative constant. Furthermore, we fix $n\es\log N$ as in Section~\ref{sec:halasz} and set 
\begin{equation*}
 q=n^{\eps},
 \quad
 \rho=q^{1/2}n^{-1},
 \quad 
 \trho=aq^{b}n^{-1}=aq^{b-1/2}\rho,
 \qquad 
 a,b,\eps>0.
\end{equation*}
As a matter of fact, we work with $q$ as if it was an integer, since its fractional part is of negligible size. Moreover, the proof of (\ref{eqn:psil1}) 
dictates $b<1/4$. Here, too, we shall continue our calculations with $b=1/4$ ,wich is compensated for by using a strict inequality sign for $\eps$ in Theorem~\ref{thm:main}. Additionally, we partition the set $\{1,2,\ldots,n\}$ into $q$ equal parts $I_1,\ldots,I_q$ , $I_v=\{(v-1)n/q+1,(v-1)n/q+2,\ldots,vn/q\}$, and group  hyperbolic vectors  into collections $\dsA_v$, $1\leq v \leq q$, according to their first coordinate:
\begin{equation*}
 \dsA_v:=\left\{\vr=(r_1,r_2,r_3)\in\dsH_n:~r_1\in I_v\right\}.
\end{equation*}
The Riesz product we intend to consider is now defined as
\begin{equation*}
 \Psi=\prod_{v=1}^{q}(1+\trho F_v)=1+\Psisd+\Psinsd,\qquad F_{v}=\sum_{\vr\in\dsA_{v}}f_{\vr},
\end{equation*}
where $\Psisd$ comprises the sums of products of strongly distinct collections of $\rsf$-functions and $\Psinsd$ contains  the rest.

The main ingredient of the proof of Theorem~\ref{thm:main} is the  lemma below.
\begin{lemma}
 \label{lemma:bilykmainlemma}
 One has the following estimates:
 \begin{align}
 \label{eqn:psil1}
  \|\Psi \|_1 &\ls 1, \\
  \label{eqn:psinsdl1}
  \|\Psinsd\|_1&\ls1, \\
  \label{eqn:psisdl1}
  \|\Psisd\|_1&\ls1,
 \end{align}
where we require $b<1/4$, and $\eps<\min\{1/3,1/(1+12b)\}$ for (\ref{eqn:psil1}) and $\eps<(8-\sqrt{41})/23$ for (\ref{eqn:psinsdl1}) and (\ref{eqn:psisdl1}), respectively.
\end{lemma}
The plain proof  of this lemma without the bounds for $\eps$ requires Littlewood-Paley theory, properties of exponential Orlicz classes as well as conditional expectation arguments and can be found in \cite{BilOnt083}. For a detailed derivation of the bound for $\eps$ for (\ref{eqn:psil1}) the reader is encouraged to study the author's PhD-thesis \cite{PucDis17}. The proof of (\ref{eqn:psinsdl1}) is dealt with in Section~\ref{sec:main}.

Let us remark that by choosing $\Psisd$ as our auxiliary function the product rule (Lemma~\ref{lemma:productrule}) is applicable and, hence, similar  arguments as those used in Section~\ref{sec:halasz} (see \cite{BilOnt083}) for relatively moderate values of $\eps$ (see \cite{PucDis17}) lead to the estimate
\begin{equation*}
 \langle D_N,\Psisd\rangle \gs a q^{b}n \es (\log N)^{1+\eps/4}.
\end{equation*}
Thus, considering (\ref{eqn:psisdl1})  we obtain our main result by Hölder's inequality. Notice that the value of $\eps$ directly determines that of $\eta$ in Theorem~\ref{thm:main}. Therefore, it is essential to meticulously keep trace of $\eps$ while proving (\ref{eqn:psinsdl1}).

\section{A brief note on the Littlewood-Paley inequalites}
\label{sec:lp}
Roth's proof heavily relies on Parseval's identity and orthogonality which, of course, are distinctive features of $L^{2}$. To apply similar methods in other functions spaces -- above all, in $L^{p}$ spaces with $1<p<\infty$ -- we require a powerful tool from harmonic analysis, the so-called \emph{Littlewood-Paley inequalities}. Since most of the proofs from this paper extensively make  use of these inequalities, this section is dedicated to provide a brief introduction to this topic tailored to our requirements. More information can be found  in~\cite{BurSha87,SteHar93,WanSha91}, for instance. 

Let us consider the case $d=1$ first. For suitable functions $f$ defined on the unit interval the \emph{dyadic square function} is given by
\begin{equation*}
 \cS f=\Bigg[| \dsE f|^2+\sum_{k=0}^{\infty}\Bigg( \sum_{J\in\cD,|J|=2^{-k}} \frac{\langle f, h_{J}\rangle}{|J|}h_{J}\Bigg)^{2} \Bigg]^{1/2}.
\end{equation*}
If we choose $f$ to be of the form $f=\sum_{J\in\cD}\alpha(J)h_{J}$ this simplifies to
\begin{equation*}
 \cS f=\Bigg[ \sum_{k=0}^{\infty}\Bigg( \sum_{J\in\cD, |J|=2^{-k}} \alpha(J)h_{J} \Bigg)^{2}  \Bigg]^{1/2}=\Bigg[ \sum_{J\in\cD}\alpha^{2}(J)\mathds{1}_{J} \Bigg]^{1/2}.
\end{equation*}
Observe that Parseval's identity may be reformulated as $\|f \|_{2}=\| \cS f\|_2$. Hence, the Littlewood-Paley inequalities as stated in the proposition below (cf. \cite{WanSha91}) can be seen as its extension to other $L^{p}$ spaces.

\begin{proposition}[Littlewood-Paley inequalites]
 For all $1<p<\infty$ there exist positive constants $A_p\geq 1+1/\sqrt{p-1}$ and $B_p\ls\sqrt{p}$ for $p\geq2$ such that
 \begin{equation*}
  A_p\|\cS f\|_{p}\leq \|f\|_{p}\leq B_p\|\cS f\|_P.
 \end{equation*}
\end{proposition}
The key observation is given by the fact that there is a version of the Littlewood-Paley inequalities (with exactly the same constants $A_p$ and $B_p$) which is valid for Hilbert space-valued functions, where the integrals involved are understood as Bochner integrals. This version allows us to apply the Littlewood-Paley inequality in, say, the first coordinate while keeping the other coordinates fixed in the sense of vector-valued coefficients. For full details of this discussion and for an illustrative example referring to Roth's proof the reader is once again advised to consult \cite{BilOnr14,BilOnt083}.

\section{The study of coincidences of hyperbolic vectors}
\label{sec:main}

The structure of coincidences within collections of hyperbolic vectors can probably be best explained by two-colored graphs. These are triples $G=(V(G),E_2,E_3)$, where $V(G)\subseteq\{1,2,\ldots,q\}=:[q]$ denotes the set of vertices and the symmetric subsets of $V(G)\times V(G)\setminus\{(k,k):k\in V(G)\}$, $E_2$ and $E_3$,  are the edge sets of color 2 and 3, respectively. Additionally, we say that $Q\subseteq V(G)$ is a \emph{clique of color} $j$ iff it is subject to
 \begin{equation*}
  \forall v,w\in Q,~v\neq w:~(v,w)\in E_j
 \end{equation*}
and $Q$ is maximal with this property. Here, maximality is understood in the sense that if $\tilde{Q}\supseteq Q$ is subject to the above condition, then $\tilde{Q}=Q$. Notice that edges serve to indicate that two vectors have a coincidence and its color states the coordinate. Hence, vertices from one clique of, say, color 2 shall correspond to a collection of hyperbolic vectors which have a coincidence in the second coordinate.
\begin{definition}
  \label{def:admissible}
 A two-colored graph $G$ is called \emph{admissible} if the following four conditions are fulfilled:
 \begin{enumerate}[(i)]
  \item Each $E_j$ decomposes into a union of cliques,
  \item If $Q_2$ and $Q_3$ are cliques of color 2 and 3, respectively, then $|Q_2\cap Q_3|\in\{0,1\}$.
  \item Every vertex is contained in at least one clique.
  \item Cliques of the same color are disjoint.
 \end{enumerate}
 Moreover, we subdivide the class of admissible \emph{connected}  (a.c.) graphs on a given vertex set $V$ further into $\cT(V)$ and $\cC(V)$. Here, $\cT(V)$ comprises all a.c.\,graphs $G$ defined on $V$ such that either
 \begin{enumerate}[(i)]
  \item $G$ is a tree or
  \item if $G$ contains a cycle then this cycle is composed of edges of one color only,
 \end{enumerate}
   and $\cC(V)$ contains the rest. That is, graphs in  $\cC(V)$ contain cycles composed of edges of both colors. We shall refer to such cycles as \emph{bicolored}.
\end{definition}
Observe that if we regard the individual cliques as vertices themselves, the elements of $\cT(V)$ admit of a tree representation. This is why we refer to them as \emph{generalized trees} in all that follows.

A bound for the number of admissible graphs on a given vertex set is given in the lemma below.
\begin{lemma}
 \label{lemma:numberofgraphs}
 Let $V\subseteq[q]$. The number of admissible graphs on $V$ is bounded by $c|V|^{2|V|}$, $c>0$. For generalized tree graphs this number reduces to $2^{|V|}|V|^{|V|-2}$.
\end{lemma}
\begin{proof}
The first bound  is derived in \cite[p. 144]{BilOnr14} and the estimate for generalized trees  is better known as \emph{Cayley's  formula} without the additional factor $2^{|V|}$ which arises from choosing one of two colors for each edge. Since elements of $\cT(V)$ can deviate from actual trees in a prescribed manner only (see item (iv) of Definition~\ref{def:admissible}) this estimate continues to hold for generalized trees. Cayley's formula was initially shown by C.W. Borchardt. Four more  recent proofs can be found in the book~\cite{AigPro14}, for instance. 
\end{proof}
The connection to our problem can now be drawn via the functions
$$
 \sprod(\dsX(G))=\sum_{(\vr_{1},\vr_{2},\ldots,\vr_{{|V|}})\in\dsX(G)}f_{\vr_{1}}\cdots f_{\vr_{|V|}},
$$
where $G$ is an admissible graph, and
\begin{equation*}
 \dsX(G):=\Big\{(\vr_{1},\vr_{2},\ldots,\vr_{{|V|}})\in\prod_{v\in V}\dsA_v:~(v_1,v_2)\in E_j \Rightarrow r_{v_1}^{(j)}=r_{v_2}^{(j)}\Big\}.
\end{equation*}

The norms of these functions for a.c.\,graphs can be estimated as follows.
\begin{lemma}
 \label{lemma:beckgain}
 Let $G$ be an a.c.\,graph with vertex set $V$, $|V|\geq2$, comprising exactly $t$ disjoint bicolored cycles. For all $\eps<1/3$ and all $1\leq l\leq q$ we have
 \begin{align*}
  \trho^{|V|}\|\sprod(\dsX(G))\|_{lq^{1/2}}
  &\ls
  \min\Big\{
  l^{\frac{3}{2}}qn^{-\frac{1}{2}},
  l^{\frac{|V|}{2}}n^{1-\frac{|V|}{2}}
  \Big\}
  l^{-\frac{t}{2}} q^{\frac{t}{4}}n^{-\frac{t}{2}}\\
  &=: M_{|V|,l} l^{-\frac{t}{2}} q^{\frac{t}{4}}n^{-\frac{t}{2}}
 \end{align*}
\end{lemma}
\begin{proof}
 In \cite{BilOnt083} Bilyk and Lacey derive an algorithm for estimating the above norm. In short, they repeatedly apply the Littlewood-Paley inequality and/or the triangle inequality to successively specify all hyperbolic vectors. As $G$ is connected, some vertices might have one or all of its coordinates fully determined even earlier. 
  
  To provide a clearer picture of their argument let us consider the graph $G_0$ on three vertices associated to the first picture in Figure~\ref{fig:graphs}.
\begin{figure}[H]
  \begin{equation*}
  \begin{matrix}
   \vr&&\vs&&\vt \\ \hline
   r_1 && s_1 && t_1 \\
   r_2&=&s_2 &\neq& t_2 \\
   r_3 &\neq&s_3&=&t_3 
  \end{matrix}
  \qquad \qquad 
  \begin{matrix}
   \vr&&\vs&&\vt \\ \hline
   r_1 && s_1 && \mu\\
   r_2&=&s_2 &\neq&\nu \\
   r_3 &\neq&n-\mu-\nu&=& n-\mu-\nu
  \end{matrix}
 \end{equation*}
 \caption{Hyperbolic vectors associated to the graphs $G_0$ (left) and $\tilde{G}_0$ (right).}
 \label{fig:graphs}
\end{figure}
W.l.o.g. assume $t_1\in I_1$.  One application of the Littlewood-Paley inequality in the first coordinate  yields
\begin{multline*}
 \| \sprod(\dsX(G_0)) \|_{lq^{1/2}}=\Bigg\|  \sum_{(\vr,\vs,\vt)\in\dsX(G_0)}f_{\vr}f_{\vs}f_{\vt}\Bigg\|_{lq^{1/2}}\\
 \ls l^{\frac{1}{2}}q^{\frac{1}{4}} \Bigg\| \Bigg[  \sum_{\mu\in I_1} \Bigg| \sum_{\substack{(\vr,\vs,\vt)\in\dsX(G_0) \\ \vt=(\mu, t_2,t_3)}} f_{\vr}f_{\vs}f_{\vt}  \Bigg|^{2} \Bigg]^{\frac{1}{2}}  \Bigg\|_{lq^{1/2}}.
\end{multline*}
 Subsequently, we fix $t_2$ with the help of the triangle inequality
 \begin{equation*}
   \| \sprod(\dsX(G_0)) \|_{lq^{1/2}} \ls l^{\frac{1}{2}}q^{\frac{1}{4}} \sum_{\nu=1}^{n}  \Bigg\| \Bigg[  \sum_{\mu\in I_1} \Bigg| \sum_{\substack{(\vr,\vs,\vt)\in\dsX(G_0) \\ \vt=(\mu, \nu,t_3)}} f_{\vr}f_{\vs}f_{\vt}  \Bigg|^{2} \Bigg]^{\frac{1}{2}}  \Bigg\|_{lq^{1/2}}.
 \end{equation*}
Notice that $\vt$ is already fully specified, since its coordinates add up to $n$. Consequently, we can pull $f_{\vt}$ out of the sum, where it simplifies to 1.  Taking the supremum w.r.t. $\mu$ and $\nu$ then finally yields
\begin{equation*}
    \| \sprod(\dsX(G_0)) \|_{lq^{1/2}}  \ls l^{\frac{1}{2}} q^{-\frac{1}{4}}n^{\frac{3}{2}} \sup_{\mu,\nu}\|\sprod(\dsX(\tilde{G}_0)) \|_{lq^{1/2}},
\end{equation*}
 where we used $|I_1|=n/q$. The vectors from $\dsX(\tilde{G}_0)$ are depicted in the right picture of Fig.~\ref{fig:graphs}. Observe that we only need to carry out the first of the above steps, i.e. the Littlewood-Paley inequality, in order to completely determine $\vs$.  We continue in this direction until we have considered every vertex as then the	 expression in modulus equals to 1.
 
 Our approach works in the following way. We make direct use of the discussion above and subsequently distinguish between a.c.\,graphs either belonging to $\cT$ or to $\cC$. In the first case we apply both the Littlewood-Paley and the triangle inequaliy at a cost of $l^{1/2}q^{-1/4}n^{3/2}$ once, thus fully specifying one vertex and simultaneously fixing one coordinate of an adjacent vector. For this vertex, in turn, we only need to apply the Littlewood-Paley inequality. In doing so we save an entire power of $n$ in each step. 
 This gives the second entry from the minimum from the claim. The first entry is a revised version of the case $|V|=2$ from \cite{BilOnt083}, which (by \cite{PucDis17}) is valid for $\eps<1/2$. If, additionally, $G$ contains a cycle with edges of different color, i.e. $G\in\cC$, then, due to the hyperbolic assumption, there is at least one vertex which is fully specified by the other vertices of the cycle. Consequently, we gain a factor of $l^{1/2}q^{-1/4}n^{1/2}$ for each of the $t$ bicolored cycles. For further details the reader is referred to \cite{PucDis17}.
\end{proof}
Now, our strategy becomes more visible. While a.c.\,graphs with bicolored cycles are hard to handle combinatorically speaking, they yield much better estimates in terms of Lemma~\ref{lemma:beckgain} compared to graphs from $\cT$. As it turns out, generalized trees account for the lion share in our estimates. To see this, we adhere to the approach of Bilyk and Lacey once again and, additionally, keep trace of $\eps$ to find that
\begin{equation}
 \label{eqn:sumadm}
 \|\Psinsd\|_1\ls\sum_{v=2}^{q} \sum_{G \text{ admissible}, |V(G)|=v}\trho^{v}\|\sprod(\dsX(G))\|_{q^{1/2}}\qquad\text{for all }\eps<1/4.
\end{equation}
One can easily check that the summands for $v=2,3$ are bounded by an absolute constant if $\eps<1/6$.

In what follows we abbreviate ${[q] \choose v}=\{V\subseteq [q]:|V|=v\}$ as well as $\cV(V,l)=\{\mathbf{V}=(V_1,V_2,\ldots,V_l):V_j\neq\emptyset\text{ and } \mathbf{V} \text{ is a partition of } V\}$. Observe that the cardinality of the above set is given by the \emph{Stirling number of the second kind}, which is known to satisfy  $ \#\cV(l,V)\ls {|V| \choose l} l^{|V|-l}$, see \cite{RenOns69}.

Let us continue with the remaining sum in (\ref{eqn:sumadm}). We decompose each admissible graph $G$ into a union of its a.c.\,components. Due to a reduction lemma from \cite{BilOnt083} we know that for any two such subgraphs $G=G_1\cup G_2$ we have $\sprod(\dsX(G))=\sprod(\dsX(G_1))\sprod(\dsX(G_2))$. Subseqently, we may apply Hölder's inequality, which yields
\begin{align}
  \|\Psinsd\|_1
  &\!\ls\!\!
  \sum_{v=4}^{q}	 \!\sum_{V\in{[ q] \choose v }}\!	\sum_{l=1}^{v/2}	\sum_{\substack{(V_1,\ldots, V_l)\in\mathcal{V}(V,l)}}	\sum_{\substack{G=G_1\cup\cdots\cup G_l \\ G_j \text{ is a.c.\,on }V_j }} \prod_{j=1}^{l} \trho^{|V_j|}\left\| \sprod(\dsX(G_j))\right\|_{lq^{1/2}}\nonumber\\
  \label{eqn:streescycle}
  &=:
  \sum_{v=4}^{q}	 \sum_{V\in{[ q] \choose v }}	\sum_{l=1}^{v/2}	\sum_{\substack{(V_1,\ldots, V_l) \in\mathcal{V}(V,l)}}
  \left( \Stree+\Scycle\right),
  \end{align}
  where
  \begin{align*}
       \Stree
 &=
 \sum_{\substack{G=G_1\cup\cdots\cup G_l \\ G_j \in \cT(V_j)}} \prod_{j=1}^{l}\trho^{|V_j|}\left\| \sprod(\dsX(G_j))\right\|_{lq^{1/2}}\quad\text{and}\\
  \Scycle
  &=
 \sum_{\substack{G=G_1\cup\cdots\cup G_l \\ G_j \text{ a.c.\,on }V_j  \text{ and } \exists j_0:T(G_{j_0})\geq1}} \prod_{j=1}^{l}  \trho^{|V_j|}\left\| \sprod(\dsX(G_j))\right\|_{lq^{1/2}}.
 \end{align*}
with $T(G_j)=\max\{\tau:G_j \text{ contains } \tau \text{ disjoint bicolored cycles}\}$.

Before we proceed with the estimation we shall give one more technical lemma.
\begin{lemma}
 \label{lemma:lagrange}
  Let $l$, $k$, and $v$ be integers with $1\leq k\leq l\leq v/2$. Furthermore, consider $v_1,v_2,\ldots,v_l\in\bN$ with $v_j\geq2$, $1\leq j\leq l$, and $v_1+v_2+\cdots+v_l=v$. Then
 \begin{equation*}
  \left(\prod_{j=1}^{k} v_j^{ v_j-2}\right)\cdot\left(\prod_{j=k+1}^{l}v_j^{2v_j} \right)\ls\left(  \frac{v}{k}\right)^{v-2k},
 \end{equation*}
and if $k=0$, i.e. the first product vanishes, we obtain $(v/l)^{2v}$ as an upper bound.
\end{lemma}
\begin{proof}
 We confine ourselves to the case where $k\geq1$, since the other case follows the same spirit. Let us consider the Lagrangian
 \begin{equation*}
  \mathcal{L}(v_1,\ldots,v_l;\lambda) =\sum_{j=1}^{k}(v_j-2)\log v_j+2\sum_{j=k+1}^{l}v_{j}\log v_j-\lambda(v_1+\cdots+v_l-v).
 \end{equation*}
 Simple algebraic manipulations lead to the solution
 \begin{gather}
 v_j = \frac{2}{w}, \quad1\leq j\leq k, \qquad \qquad v_j=e^{\lambda/2-1},\quad k<j\leq l,\nonumber\\
 \label{eqn:lagrangesolutions}
 \lambda= \frac{2}{3}\left(2+\log\left(\frac{wv-2k}{e^{1-\lambda} w(l-k)}\right)\right) =  \frac{2}{3}\left(2+\log\left(\frac{ve^{1-\lambda}-ke^w}{e^{2-2\lambda}(l-k)}\right)\right),
\end{gather}
where $w=W(2e^{1-\lambda})$ with $W$ denoting the \emph{Lambert $W$ function}. Observe that $0<w\leq1$ since $v_j\geq2$, and, consequently,
\begin{equation*}
 v_j=\frac{2}{w}\leq \frac{2}{w}e^{1-w}=\frac{2e}{w e^w}=e^{\lambda},\quad 1\leq j\leq k.
\end{equation*}
Furthermore, if $l-k>\sqrt{kv}$ we immediately get $v_j=e^{\lambda/2-1}\leq \frac{v}{l-k}<\sqrt{v/k}$ for $\quad k<j\leq l
$ by (\ref{eqn:lagrangesolutions}). On the other hand, if  $l-k\leq\sqrt{kv}$ we can solve the last expression in (\ref{eqn:lagrangesolutions}) for $\lambda$, giving
\begin{equation*}
  \lambda=2 \log\left(\frac{l-k+\sqrt{(l-k)^2+4e^{w+1}kv}}{2ke^{w}}\right),
\end{equation*}
and, thus, $e^{\lambda/2}\ls\sqrt{v/k}$ can be obtained without difficulty. Consequently, the left-hand side from the claim can be estimated by
\begin{equation*}
 e^{\lambda(v_1+\cdots+v_k-2k)} e^{2(\frac{\lambda}{2}-1)(v_{k+1}+\cdots+v_l)}
 \ls
  e^{\lambda (v_1+\cdots +v_l)-2\lambda k }
 \!=\!
 e^{\lambda(v-2k)}
 \!\ls\!
 \left(\frac{v}{k}\right)^{v-2k}\!.
\end{equation*}
\end{proof}
Within the subsequent paragraphs we show
\begin{equation}
\label{eqn:stree}
 \sum_{v=4}^{q}\sum_{V\in{[q]\choose v}}\sum_{l=1}^{v/2}\sum_{(V_1,\ldots,V_l)\in \cV(V,l)}\Stree \ls1\qquad\text{for all }\eps<(8-\sqrt{41})/23.
\end{equation}
Indeed, for all $\eps<1/3$ we have by Lemmas~\ref{lemma:numberofgraphs}, \ref{lemma:beckgain}, and \ref{lemma:lagrange}
\begin{equation}
\label{eqn:streem}
 \Stree
 \ls
 \prod_{j=1}^{l}M_{|V_j|,l} |V_j|^{|V_j|-2}\ls l^{-v+2l}v^{v-2l}\prod_{j=1}^{l}M_{|V_j|,l}.
\end{equation}
Let us choose $\at\in(0,1/2)$ arbitrarily for now and consider the sum over $l$. For the first $\at v$ summands we choose the second entry of the minimum $M_{|V_j|,l}$ and the first entry for all the others. Since $|V_1|+\cdots+|V_l|=v$ this yields
\begin{align*}
 \sum_{l=1}^{v/2}\sum_{(V_1,\ldots,V_l)\in \cV(V,l)}\!\!\!\!\!\Stree 
 &\ls
 \sum_{l=1}^{\at v}\! {v \choose l}\! l^{\frac{v}{2}+l}v^{v-2l}n^{-\frac{v}{2}+l}\!+\!\!\sum_{l=\at v+1}^{v/2}\!\!\! {v \choose l} l^{\frac{5}{2}l}v^{v-2l}q^{l}n^{-\frac{l}{2}}\\
 &=:\Sti +\Stii.
\end{align*}
By Stirling's formula we immediately obtain
\begin{equation}
\label{eqn:sti}
 \Sti\ls \sum_{l=1}^{\at v} l^{\frac{v}{2}-\frac{1}{2}}v^{v-l}n^{-\frac{v}{2}+l}
 \leq v^{\frac{3}{2}-\frac{1}{2}}n^{-\frac{v}{2}}\sum_{l=1}^{\at v}(v^{-1}n)^{l}
 \ls v^{v(\frac{3}{2}-\at)-\frac{1}{2}}n^{-v(\frac{1}{2}-\at)}.
\end{equation}
For the estimation of $\Stii$ we observe that
\begin{equation*}
 {v\choose l\!+\!\at v\!+\!1}\!\leq\! {v\!-\!\at v\!-\!1\choose l}(l+\at v +1)^{-\at v-1}v^{\at v+1}\! \ls\! {v\!-\!\at v\!-\!1\choose l}
\end{equation*}
and, consequently,
\begin{align}
 \Stii 
 &\ls
 v^{v(1+\frac{\at}{2})-\frac{1}{2}}q^{\at v+1}n^{-\frac{\at}{2}v-\frac{1}{2}}\sum_{l=0}^{\frac{v}{2}-\at v-1}{v-\at v-1\choose l}(v^{\frac{1}{2}}qn^{-\frac{1}{2}})^{l}\nonumber\\
 \label{eqn:stii}
 &\ls v^{v(1+\frac{\at}{2})+\frac{1}{2}}q^{\at v+1} n^{-\frac{\at }{2}v-\frac{1}{2}},
\end{align}
since $(1+v^{1/2}qn^{-1/2})^{v-\at v-1}\ls \exp(v^{3/2}q n^{-1/2})\ls1$ for $\eps<1/5$.

In the same spirit we may now derive (\ref{eqn:stree}):
\begin{align*}
& \sum_{v=4}^{q}\sum_{V\in{[q]\choose v}}(\Sti+\Stii)\\
  &\ls\!\sum_{v=0}^{q-4}\!{q\!-\!4\choose v} \!\!\left( \!\!q^{v(\!\frac{3}{2}-\at\!)+\frac{11}{2}-4\at} \! n^{-v(\!\frac{1}{2}-\at\!)-2+4\at}\!\! +\! q^{v(\!1+\frac{3\at}{2}\!)+\frac{11}{2}+6\at}n^{-\frac{\at}{2}v-\frac{1}{2}-2\at}  \!\!\right)\\
  &\ls
  q^{\frac{11}{2}-4\at}n^{-2+4\at}e^{q^{5/2-\at}n^{-1/2+\at}}+q^{\frac{11}{2}+6\at}n^{-\frac{1}{2}-2\at}e^{q^{2+3\at/2}n^{-\at/2}}.
\end{align*}
The latter expression is bounded by a constant if
\begin{equation*}
 \eps<\epst(\at):=\min\left\{\frac{4-8\at}{11-8\at},\frac{1-2\at}{5-2\at},\frac{1+4\at}{11+12\at},\frac{\at}{4+3\at}   \right\}.
\end{equation*}
It is fairly easy to see that the optimal $\at\in(0,1/2)$ is $\ato=(\sqrt{41}-5)/4$, for which we have $\epst(\ato)=(8-\sqrt{41})/23$, and (\ref{eqn:stree}) follows.

What is left to show is that the part of (\ref{eqn:streescycle}) comprising $\Stree$ outweighs the part with $\Scycle$. To this end we notice that  a bicolored cycle can only occur if it contains at least four vertices and we may thus estimate
\begin{align*}
 \Scycle
 &\ls
 \sum_{t=1}^{v/4}\sum_{\substack{t1,\ldots,t_l\geq0 \\ t_1+\cdots+t_l=t}} \sum_{\substack{G_1 \text{ a.c.\,on } V_1\\T(G_1)=t_1}} \cdots \sum_{\substack{G_l \text{ a.c.\,on } V_l\\T(G_l)=t_l}}\prod_{j=1}^{l} \trho^{|V_j|}\|\sprod(\dsX(G_j))\|_{lq^{1/2}}\\
 &\ls (\Ssl+\Sgl)\prod_{j=1}^{l}M_{|V_j|,l},
\end{align*}
where we used Lemma~\ref{lemma:beckgain} with $\eps<1/3$ and where we set
\begin{align*}
 \Ssl 
 &=
 \sum_{t=1}^{l-1}l^{-\frac{t}{2}}q^{\frac{t}{4}}n^{-\frac{t}{2}} \sum_{\substack{t_1,t_2,\ldots,t_l\geq0 \\ t_1+t_2+\cdots+t_l=t}} \sum_{\substack{G_1 \text{ a.c.\,on } V_1 \\ T(G_1)=t_1}}\cdots\sum_{\substack{G_l \text{ a.c.\,on } V_l \\ T(G_l)=t_l}}1, \text{ and}\\
 \Sgl 
 &=
 \sum_{t=l}^{v/4}l^{-\frac{t}{2}}q^{\frac{t}{4}}n^{-\frac{t}{2}} \sum_{\substack{t_1,t_2,\ldots,t_l\geq0 \\ t_1+t_2+\cdots+t_l=t}} \sum_{\substack{G_1 \text{ a.c.\,on } V_1 \\ T(G_1)=t_1}}\cdots\sum_{\substack{G_l \text{ a.c.\,on } V_l \\ T(G_l)=t_l}}1.
\end{align*}
Observe that, for fixed $t$, at least $l-t$ of the subgraphs occuring in $\Ssl$ do not contain a bicolored cycle. Hence, we may apply Lemma~\ref{lemma:lagrange} with $k=l-t$, and together with Stirling's formula this yields
\begin{align}
 \Ssl
 &\ls
 \sum_{t=1}^{l-1}l^{-\frac{t}{2}}q^{\frac{t}{4}}n^{-\frac{t}{2}}{t+l-1 \choose l-1}\left(\frac{v}{l-t}\right)^{v-2(l-t)}\nonumber\\
\label{eqn:ssl}
 & \ls 
 l^{-\frac{1}{2}}v^{v-2l}\sum_{t=1}^{l-1}(l-t)^{-v+2(l-t)}\left(l^{-\frac{1}{2}}v^2q^{\frac{1}{4}}n^{-\frac{1}{2}}\right)^{t}.
\end{align}
Moreover, as $\eps\leq1/15$, each summand from the latter expression is bounded by $H(t)=(l-t)^{-v+2(l-t)}(l^{-1/2}v^{-21/4})^t$. Obviously, $H(1)\leq l^{-v+2l}$, which is the corresponding part of $\Stree$, see (\ref{eqn:streem}). Furthermore, the only critical point of $H$ is given by $t_0$ with $l-t_0=v/(2W(z_0))$, $z_0=e/2\cdot l^{1/4}v^{29/8}$. Since $z_0\geq e\cdot2^{54/4}$ we have $W(z_0)>79/20$. On the other hand, for each $\kappa>0$
there exists a constant $c>0$ such that $W(z_0)\leq z_0^{\kappa}+c$. By choosing $\kappa=1/1711$, for instance, we see that $H(t_0)$ is bounded by $H(l-1)$ from above.

Thus, it remains to investigate $H(l-1)$. For $l\geq\ato v+1$ this number is bounded by $H(1)$, which has already been dealt with. For $l\leq\ato$ the summands themselves are not well comparable to parts of $\Stree$ individually. However, in average, that is, considering everything down to the  sum over $l$, they are bounded by (\ref{eqn:sti}). Indeed, for all $\eps<4/35$ we have
\begin{multline*}
 \sum_{l=1}^{\ato v}\!\!\sum_{(V_1,\ldots,V_l)\in\cV(V,l)}  \prod_{j=1}^{l}\!M_{|V_j|,l}\Ssl 
 \ls
 \sum_{l=1}^{\ato v}\!\!\!{v\choose l} l^{v-l}l^{\frac{v}{2}}n^{-\frac{v}{2}+l}l^{-\frac{1}{2}}v^{v-2l}H(l-1)\\
 \ls
 v^{v(\frac{5}{2}-\frac{35}{4}\ato)+\frac{19}{4}}n^{-v(\frac{1}{2}-\ato)}
 \ls
 v^{v(\frac{3}{2}-\ato)-\frac{1}{2}}n^{-v(\frac{1}{2}-\ato)}.
\end{multline*}

For the study of $\Sgl$ we proceed similarly to (\ref{eqn:ssl}). We exploit Lemma~\ref{lemma:lagrange}, but now for $k=0$, and obtain
\begin{equation*}
 \Sgl \ls  l^{-2v-l+\frac{1}{2}} v^{2v+l-1}\sum_{t=l}^{v/4}\left( l^{-\frac{1}{2}} q^{\frac{1}{4}} n^{-\frac{1}{2}} \right)^{t}
 \ls l^{-2v-\frac{3}{2} l+\frac{1}{2}} v^{2v+l-1} q^{\frac{l}{4}}n^{-\frac{l}{2}}.
\end{equation*}
The rest follows more or less the same strategy as proving (\ref{eqn:stree}). I.e., we split up the sum over $l$ at $\ato v$, estimate accordingly as we did for $\Sti$ and  $\Stii$ and find that their upper bounds (\ref{eqn:sti}) and (\ref{eqn:stii}) dominate the resulting expressions.  For full details, see \cite{PucDis17}, again. This finishes the proof of Theorem~\ref{thm:main}.

\section{Acknowledgements}
The author is extremely grateful to his supervisor, Gerhard Larcher, for his valuable consultation, for proof reading  and for his general assistance and encouragement during the writing of this paper. 

\bibliography{mybib}
\bibliographystyle{plain}
\end{document}